\documentclass{article}
\usepackage[centertags]{amsmath}
\usepackage{amsfonts}
\usepackage{amssymb}
\usepackage{amsthm}
\usepackage{newlfont}
\usepackage[top=1in, bottom=1in, left=1in, right=1in]{geometry}
\usepackage{tikz}

\theoremstyle{plain}
\newtheorem{theorem}{Theorem}
\newtheorem*{theorem*}{P\'olya's Theorem}
\newtheorem{proposition}{Proposition}[section]
\newtheorem{corollary}{Corollary}[section]
\newtheorem{lemma}{Lemma}[section]

\theoremstyle{definition}

\newcommand{\N}{\mathbb{N}}
\newcommand{\Z}{\mathbb{Z}}

\newcommand{\Hidden}[1]{}

\DeclareMathOperator{\Tr}{Trace}

\begin{document}

\title{A non-backtracking P\'olya's theorem}
\author{Mark Kempton\footnote{Center of Mathematical Sciences and Applications, Harvard University, Cambridge, MA.  Email: mkempton@cmsa.fas.harvard.edu}}
\date{}

\maketitle

\begin{abstract}
P\'olya's random walk theorem states that a random walk on a $d$-dimensional grid is recurrent for $d=1,2$ and transient for $d\ge3$.  We prove a version of P\'olya's random walk theorem for non-backtracking random walks. Namely, we prove that a non-backtracking random walk on a $d$-dimensional grid is recurrent for $d=2$ and transient for $d=1$, $d\ge3$. Along the way, we prove several useful general facts about non-backtracking random walks on graphs.  In addition, our proof includes an exact enumeration of the number of closed non-backtracking random walks on an infinite 2-dimensional grid.  This enumeration suggests an interesting combinatorial link between non-backtracking random walks on grids, and trinomial coefficients.  
\end{abstract}

\section{Introduction}
P\'olya's celebrated random walk theorem, first proven by Goerge P\'olya in 1921 in \cite{polya}, characterizes the behavior of random walks on infinite grids of all dimensions.  We say that a random walk on an infinite graph is \emph{recurrent} if the random walk is guaranteed (with probability 1) to return to its starting point, and is called \emph{transient} if there is a positive probability that the random walk never returns to its starting point.
\begin{theorem*}

A simple random walk on the infinite grid $\Z^d$ is recurrent for $d=1,2$ and transient for $d\geq3$.

\end{theorem*}

Since P\'olya's paper in 1921, this theorem has become a standard example in probability theory and has motivated considerable work in random walks, including numerous generalizations and variations.  See in particular \cite{moore} and \cite{sabot}.

Our goal is to prove an analogous result for non-backtracking random walks, which are random walks with the extra condition that they are not permitted to return to the vertex from the immediately previous step.  Extensive research has been done recently in the study of non-backtracking random walks.  The convergence and mixing rate of non-backtracking random walks is studied in \cite{alon, cioaba, Fitzner} and \cite{ihara}. The distribution of the number of visits of a random walk to a vertex is studied in \cite{alon2}, and \cite{angel} studies non-backtracking random walks on the universal cover of a graph.    In \cite{redemption}, non-backtracking random walks are used to study spectral clustering algorithms.  A non-backtracking random walk on a graph is not a Markov chain when the state space is taken to be the vertex set of the graph, but can be turned into a Markov chain by thinking of the walk as moving along directed edges of the graph.  In particular, \cite{angel, ihara} and \cite{redemption} take this approach.

The main result of this paper is to characterize recurrence and transience of non-backtracking random walks on infinite grids of all dimensions (i.e., the non-backtracking version of P\'olya's theorem).  Our result is as follows.

\begin{theorem}[Non-backtracking P\'olya's theorem]\label{thm:nbp}
A non-backtracking random walk on the infinite grid $\Z^d$ is recurrent for $d=2$ and transient for $d=1$ and $d\geq3$.
\end{theorem}

Our proof involves two main sections: proving transience for $d\ge3$ and proving recurrence for $d=2$.  We remark that for the case $d=1$, it is trivial to see that a non-backtracking random walk is transient, since, after the first step, the walk is forced into a single direction.

For $d\geq3$, we are able to compare the probability of a non-backtracking walk returning to its starting point to the probability of a simple random walk returning to its starting point, which will give transience.  Along the way, we record several general facts about non-backtracking random walks that may be useful for future work in studying this topic.  In particular, Lemmas \ref{lem:recurrence}, \ref{lem:A^n}, and \ref{lem:eig_asympt}, as well as equation (\ref{eq:gen_func}) are likely to be of general interest to anyone wishing to study non-backtracking random walks.

For $d=2$, the situation is more complicated, and we produce an exact expression for the number of closed non-backtracking random walks of a given length, and analyze this expression.  In doing this, we discover an interesting combinatorial connection between non-backtracking random walks in the 2-d grid, and central trinomial coefficients.  

The remainder of this paper is organized as follows.  In section \ref{sec:prelim}, we will give background on P\'olya's theorem and aspects of its proof that will be useful to us later, as well as background on non-backtracking random walks, including some facts of general interest about them.  Section \ref{sec:d>3} will develop some general techniques for non-backtracking random walks on regular graphs, from which we will be able to deduce the main result for $d\ge3$ via comparison with simple random walks.  Finally, in section \ref{sec:Z2}, we will finish the proof of Theorem \ref{thm:nbp} by covering the $d=2$ case.  

\section{Preliminaries}\label{sec:prelim}

Let $G$ be a graph.  A \emph{random walk} on $G$ of length $k$ is a sequence of vertices $(v_0,v_1,...,v_k)$ in which the vertex $v_{i+1}$ is chosen uniformly at random from among the neighbors of $v_i$.  We will use the terminology \emph{simple random walk} to specify the usual, unrestricted random walk (as opposed to a non-backtracking random walk).  See \cite{lovasz} for a good introduction to the theory of random walks on graphs.

\subsection{P\'olya's Theorem}\label{sub:polya}

Suppose $G$ is a graph with infinitely many vertices.  Consider a random walk on $G$ starting at some initial vertex $v_0$.  The random walk on $G$ is called \emph{recurrent} if the probability that the walk eventually returns to $v_0$ is 1.  If this probability is less than one, the random walk is called \emph{transient}.  As mentioned, P\'olya's Theorem says that a simple random walk on $\Z^d$ is transient for $d\ge3$ and recurrent for $d=1,2$.

P\'olya's Theorem is well known, and numerous proofs exist in the literature.  For instance, \cite{polya} has the original proof by P\'olya, various proofs coming from the theory of electrical networks and random walks can be found in \cite{doyle, doylesnell, tetali}, and \cite{novak} has a proof involving methods from special function theory. In this section, we will sketch some the ideas of a classical proof of P\'olya's Theorem based on enumerating walks on the grid.  We mention this proof here because some of the ideas therein will be useful to us later on.

Let $p(k)$ denote the probability that a random walk returns to its starting vertex after $k$ steps.  The key to the proof of P\'olya's Theorem is the following lemma, which is well-known (see \cite{asymptopia}, for example).

\begin{lemma}\label{lem:conv}
If the sum
\[
\sum_{k=0}^\infty p(k)
\]
is convergent, then the random walk is transient.  Otherwise, it is recurrent.
\end{lemma}
The intuition behind this lemma is that $\sum_kp(k)$ represents the expected number of times a random walk returns to the origin.  If this is infinite, the random walk must always return; if it is finite, then it is possible that it does not return.  Therefore, to prove recurrence or transience of a random walk, one approach is to enumerate the total number of walks of length $k$ on the graph, and enumerate the total number of walks of length $k$ that return to the initial vertex at step $k$, then from this obtain the probability $p(k)$, and analyze the series.  

For $d=1$, it is not hard to see that there are $\binom{2n}{n}$ closed walks of length $2n$ on $\Z$, choosing $n$ steps in one direction, the other $n$ being fixed in the other direction. There are $2^{2n}$ total walks of length $2n$, so that \[p(2n) = \frac{1}{2^{2n}}\binom{2n}{n}.\] (Note that it is clear that $p(2n+1) =0$.)  Using Stirling's formula, we can see that $p(2n) \sim \frac{1}{\sqrt{\pi n}}$, so the series from Lemma \ref{lem:conv} is divergent, and thus the random walk is recurrent.

For $d=2$, in a similar manner we have
\[
p(2n) = \frac{1}{4^{2n}}\sum_{k=0}^n\frac{(2n)!}{k!k!(n-k)!(n-k)!} = \frac{1}{4^{2n}}\binom{2n}{n}\sum_{k=0}^n\binom{n}{k}^2 = \frac{1}{4^{2n}}\binom{2n}{n}^2
\]
using the basic combinatorial fact that $\sum_k\binom{n}{k}^2 = \binom{2n}{n}$.  Again, Stirling's formula gives $p(2n)\sim \frac{1}{\pi n}$ so the series from Lemma \ref{lem:conv} diverges, and the random walk is recurrent.

For $d\ge3$, similar combinatorial formulas can be obtained, and more involved manipulation of these shows that \[p(2n) \sim \frac{c}{(\pi n)^{d/2}}.\] Therefore, for $d\ge3$ the series from Lemma \ref{lem:conv} is convergent, so the walk is transient. The details will not be needed here, but can be found in many probability texts (see \cite{prob,asymptopia} for example).

\subsection{Non-backtracking Random Walks}

A non-backtracking random walk on a graph $G$ is a sequence of vertices $(v_0, v_1, ... v_k)$ such that $v_{i+1}$ is chosen randomly among the neighbors of $v_{i}$ such that $v_{i+1}\neq v_{i-1}$.  

Define the matrix $\tilde A^{(k)} = \tilde A^{(k)}(G)$ by setting $\tilde A^{(k)}(u,v)$ to be the number of non-backtracking walks of length $k$ from vertex $u$ to vertex $v$.  Let $A$ denote the adjacency matrix of $G$.  Where the entries of $A^k$ count walks on a graph, the entries of $\tilde A^{(k)}$ count non-backtracking walks on the graph.  Let $D$ be the degree matrix of $G$, that is, $D$ is diagonal with $D(v,v) = d_v$ for each vertex $v$ of $G$.  In \cite{alon}, a recurrence is given for $\tilde A^{(k)}$ for the case of regular graphs.  Here we generalize it for arbitrary graphs.

\begin{lemma}\label{lem:recurrence}
 The matrices $\tilde A^{(k)}$ satisfy the recurrence
\[\begin{cases}
\tilde A^{(1)} = A\\ \tilde A^{(2)} = A^2-D\\ \tilde A^{(k+2)} = A\tilde A^{(k+1)} - (D-I)\tilde A^{(k)}
\end{cases}\]
\end{lemma}

\begin{proof}
Since the non-backtracking condition puts no restriction on the first step, it is clear that $\tilde A^{(1)} = A$.  For $\tilde A^{(2)}$, note that $A^2$ counts all walks of length 2, so we must simply subtract off those that backtrack. The only walks of length 2 that backtrack are those that move from a vetex to a neghbor, then return immediately.  For a vertex $x$, there are clearly $d_x$ such walks, so this is what must be subtracted from the diagonal.

For ease of notation, we will use the symbol $\tilde A^{(k)}_x(u,v)$ to denote the number of non-backtracking walks of length $k$ starting at vertex $u$ and ending at vertex $v$ and that pass through vertex $x$ at step $k-1$.  To obtain $\tilde A^{(k+2)}(u,v)$, we add the number of non-backtracking walks of length $k+1$ from $u$ to each neighbor of $v$, and then subtract those that backtracked, that is, those that visited $v$ at step $k$.  More specifically,
\[\begin{split}
\tilde A^{(k+2)}(u,v) &= \sum_{x\sim v}\left( \tilde A^{(k+1)}(u,x) - \tilde A^{(k+1)}_v(u,x) \right) \\
&= \sum_{x\sim v}\left( \tilde A^{(k+1)}(u,x) - \left(\tilde A^{(k)}(u,v) - \tilde A^{(k)}_x(u,v) \right)\right) \\
& = \left[A\cdot \tilde A^{(k+1)}\right](u,v) - d_v\cdot \tilde A^{(k)}(u,v) + \tilde A^{(k)}(u,v).
\end{split}\]
It follows that \[ \tilde A^{(k+2)} = A\tilde A^{(k+1)} - (D-I)\tilde A^{(k)} \] as claimed. 
\end{proof}

For convenience we will define $\tilde A^{(0)} = I$. Define the generating function \[ F(x) = \sum_{k=0}^\infty \tilde A^{(k)}x^k, \] then from Lemma \ref{lem:recurrence} we can determine the generating function
\begin{equation}\label{eq:gen_func}
F(x) = (1-x^2)\left( I - xA+x^2(D-I)\right)^{-1}.
\end{equation}

\section{Transience for $d\ge3$}\label{sec:d>3}

\subsection{Regular Graphs}

In this section, we will assume $G$ is a regular graph, with constant degree that we will call $r$.  Note that for the infinite grid $\Z^d$, we have $r=2d$.  We will start with some general results on non-backtracking random walks for any regular graphs, and then we will apply these to grids in view of proving Theorem \ref{thm:nbp}. 

\begin{lemma}\label{lem:A^n}
For $G$ an $r$-regular graph, then using notation from the previous section, we can express the number of non-backtracking random walks on $G$ starting from $u$ and ending at $v$ of length $n$ as
\[
\tilde A^{(n)}(u,v) =  \sum_{i=0}^{\lfloor n/2\rfloor}(-1)^i\binom{n-i}{i}(r-1)^iA^{n-2i}(u,v) - \sum_{i=0}^{\lfloor n/2-1\rfloor}(-1)^i\binom{n-i-2}{i}(r-1)^iA^{n-2i-2}(u,v).
\]
\end{lemma}
\begin{proof}
From (\ref{eq:gen_func}), we can expand the expression as a geometric sum to obtain
\[ F(x) = (1-x^2)\sum_{k=0}^\infty\left(A - (D-I)x\right)^kx^k.\]

Now, we are assuming $G$ is $r$-regular, so $D-I = (r-1)I$, and the above can be further expanded, yielding
\[ F(x) = (1-x^2)\sum_{k=0}^\infty\sum_{i=0}^k (-1)^i\binom{k}{i}(r-1)^iA^{k-i}x^{k+i}.\]
Recalling that $F(x) = \sum \tilde A^{(n)}x^n$,
a general formula for $\tilde A^{(n)}$ can be obtained by extracting the $x^n$ coefficient.
\[
\tilde A^{(n)} = [x^n]F(x) = \sum_{i=0}^{\lfloor n/2\rfloor}(-1)^i\binom{n-i}{i}(r-1)^iA^{n-2i} - \sum_{i=0}^{\lfloor n/2-1\rfloor}(-1)^i\binom{n-i-2}{i}(r-1)^iA^{n-2i-2}.
\]
Taking the $u,v$ entry of the matrix gives the statement of the lemma.
\end{proof}

We remark that the expression $A^n(u,v)$ is simply the total number of walks of length $n$ from $u$ to $v$, so we have expressed the number of non-backtracking walks in terms of the total number of simple walks.

For $r$-regular graphs, it will be useful to work directly with the transition probability matrix.  Denote $P = D^{-1}A$, which is the transition probability matrix for a simple random walk on $G$, so that $P^k(u,v)$ is the probability that a simple random walk starting a vertex $u$ ends at vertex $v$ after $k$ steps.  Similarly, we will let $\tilde P^{(k)}$ denote the transition probability matrix for a non-backtracking random walk on $G$.  That is, $\tilde P^{(k)}(u,v)$ is the probability that a non-backtracking random walk starting at $u$ will be at $v$ after $k$ steps.  Observe that, because of the non-backtracking condition, it is not that case that $\tilde P^{(k)}$ is simply the $k$th power of $\tilde P^{(1)}$.  Indeed, a non-backtracking random walk, viewed on the vertices, is not a Markov chain since at each step, we must remember the step taken previously.  We will let $\tilde P^{(0)} = I$, and observe that $\tilde P^{(1)} = P$ since the non-backtracking restriction does not apply to the first step. Then it is straightforward to adapt Lemma \ref{lem:recurrence} to obtain a recurrence relation for $\tilde P^{(k)}$, namely 
\[\begin{cases}
\tilde P^{(0)} = I\\ \tilde P^{(1)} = P\\ \tilde P^{(k+2)} = \frac{r}{r-1}P\tilde P^{(k+1)} - \frac{1}{r-1}\tilde P^{(k)} & k\ge2.
\end{cases}\]

As a side note, we remark that, although a non-backtracking random walk on the vertex set of a graph is not a Markov chain (as noted above), it is possible to turn a non-backtracking random walk into a Markov chain by changing the state space to the set of directed edges of the graph (one directed edge in each direction for each edge of the graph), thus viewing the walk as moving along these directed edges.  The current paper will not need this technique, as the Markov property will not be important in what we do.  We refer the reader to \cite{ihara} for a thorough exposition of this point of view.

In order to understand the matrix $\tilde P^{(k)}$ via the above recurrence, we will study the polynomials that are defined by the same recurrence.  That is, let us define the polynomials $p_{r,k}(x) = p_k(x)$ according to the recurrence relation
\[\begin{cases}
p_0(x) = 1\\ p_1(x) = x\\ p_{k+2}(x) = \frac{r}{r-1}xp_{k+1}(x) - \frac{1}{r-1}p_{k}(x) & k\ge2
\end{cases}\]
and observe that $\tilde P^{(k)} = p_k(P)$, the polynomial evaluated at the matrix $P$.  It follows therefore that if $\lambda$ is an eigenvalue of $P$, then $p_k(\lambda)$ is an eigenvalue of $\tilde P^{(k)}$.

As a side note, we observe that for $r=2$, the $p_k$'s become the well-known Chebyshev polynomials of the first kind.  Thus we may view these polynomials as a sort of generalized Chebyshev polynomial.

\begin{lemma}\label{lem:eig_asympt}
Let $\lambda$ be an eigenvalue of $P$.  Then $p_k(\lambda)$ is an eigenvalue of $\tilde P^{(k)}$ and satisfies
\begin{align*}
|p_k(\lambda)| &~~~\le~~~|\lambda|^k &\text{ for }~~~ |\lambda|>\frac{2\sqrt{r-1}}{r}\\
|p_k(\lambda)| &~~~\le~~~C_rk\left(\frac{1}{\sqrt{r-1}}\right)^k &\text{ for }~~~ |\lambda| \le \frac{2\sqrt{r-1}}{r}
\end{align*}
for $r\ge2$, and for all $k$, where $C_r$ is a constant depending only on $r$ (not $\lambda$ or $k$).
\end{lemma}

\begin{proof}
Of course, we can obtain an explicit expression for $p_k$ as a polynomial in the same way that we found the formula in Lemma \ref{lem:A^n}, however, this will not be as illuminating in terms of the asymptotics of the probabilities.  Instead, we will apply techniques from the theory of difference equations to obtain an alternative explicit expression for $p_k$.  Note that the recurrence for $p_k(x)$ has characteristic equation 
\[
p^2 - \frac{r}{r-1}xp - \frac{1}{r-1}
\]
whose roots are
\[
\frac{\frac{r}{r-1}x\pm\sqrt{\frac{r^2}{(r-1)^2}x-4}}{2} = \frac{rx\pm\sqrt{r^2x^2-4(r-1)}}{2(r-1)}.
\]

To make this easier to write down, let us denote $\Delta = r^2x^2 - 4(r-1)$ for the discriminant of the characteristic equation for the recurrence, and let $\theta$ denote an angle satisfying $\cos\theta = rx/2\sqrt{r-1}$.  Then standard techniques for solving difference equations lead us to
\begin{equation}\label{eq:p(x)}
\begin{array}{rlr}
p_k(x) &= \frac{1}{2}\left[\left(1+x\frac{r-2}{\sqrt{\Delta}}\right)\left(\frac{rx+\sqrt{\Delta}}{2(r-1)}\right)^k + \left(1-x\frac{r-2}{\sqrt{\Delta}}\right)\left(\frac{rx-\sqrt{\Delta}}{2(r-1)}\right)^k\right] & \text{for } |x| > \frac{2\sqrt{r-1}}{r}\\
p_k(x) &= \left(\frac{1}{\sqrt{r-1}}\right)^k\left(\cos(\theta k) + \frac{(r-2)x}{\sqrt{-\Delta}}\sin(\theta k)\right) & \text{for } |x| < \frac{2\sqrt{r-1}}{r}\\
p_k(x) &= \left(\frac{1}{\sqrt{r-1}}\right)^k\left(1+\frac{r-2}{r}k\right) & \text{for } |x| = \frac{2\sqrt{r-1}}{r}.
\end{array}
\end{equation}

Each of these follows from a straightforward induction argument.  

Let $\lambda$ be an eigenvalue of $P$ and observe that since $P$ is a transition probability matrix for a random walk, then we have $-1\leq\lambda\leq1$.  Direct computation show us that
\[
\left|\frac{r\lambda\pm\sqrt{r^2\lambda^2-4(r-1)}}{2(r-1)}\right| \leq |\lambda| \, \text{ for } \frac{2\sqrt{r-1}}{r} < |\lambda| \leq1.
\]
Therefore, (\ref{eq:p(x)}) gives the lemma for this case.

For the other case, (\ref{eq:p(x)}) immediately gives the lemma for $|\lambda| = 2\sqrt{r-1}/r$.  For $|\lambda| < 2\sqrt{r-1}/r$, looking at (\ref{eq:p(x)}), we observe that the $\cos(\theta k)$ and $\sin(\theta k)$ terms of course remain bounded, so only problem that could arise would be the fraction becoming unbounded as $\lambda$ approaches $2\sqrt{r-1}/r$, (since the denominator, $\sqrt{-\Delta}$, goes to 0).  However, observe that since $\cos\theta = rx/2\sqrt{{r-1}}$, computation gives $\sin\theta = \sqrt{-\Delta}/2\sqrt{r-1}$.  Basic calculus gives us that as $\theta$ goes to 0, $\sin(\theta k)$ becomes asymptotic to $k\sin\theta$.  Therefore we obtain the result.

\end{proof}

\subsection{Grids}
In this section, we will apply the tools that we have developed to grids.  Fix a dimension $d\geq3$ and observe that $\Z^d$ is an infinite $r$-regular graph with $r = 2d$.  Let $p(k)$ denote that probability that random walk on $\Z^d$ returns to its starting point after $k$ steps, and $\tilde p(k)$ the probability that a non-backtracking random walk on $\Z^d$ returns to its starting point after $k$ steps.  Note that on a grid, it is clear that if $k$ is odd, then we have $p(k) = \tilde p(k) =0$ since any walk that returns to its starting point must have an equal number of steps going away from the starting point as towards it in any of the coordinate directions.  Thus we need only concern ourselves with even length walks.  That is, we are only concerned with $p(2k)$ and $\tilde p(2k)$.

\begin{proposition}\label{prop:series}
With the notation defined above, we have
\[
\sum_{k=0}^\infty \tilde p(2k)
\]
is convergent.
\end{proposition}
\begin{proof}
From Lemma \ref{lem:conv}, since the random walk on $\Z^d$ is transient, we have that 
\[
\sum_{k=0}^\infty p(2k) < \infty.
\]
We will proceed by comparing $\tilde p(2k)$ to $p(2k)$ by way of the tools from the last section.

In order to take advantage of eigenvalues, rather than look at the walks on the infinite grids, we will look at a finite quotient of the grid, namely the $d$-dimensional torus.  We will let $T_n^{(d)}$ denote the $d$-dimensional torus of length $n$. More specifically, let $C_n$ denote the cycle on $n$ vertices.  Then $T_n^{(d)}$ is formed by taking the cartesian product $C_n\Box C_n$ of the cycle with itself $d$ times.  Locally, $T_n^{(d)}$ is indistinguishable from $\Z^d$.  In particular, $T_n^{(d)}$ is still a $2d$-regular graph.  Define $p_{T_n^{(d)}}(2k)$ to be the probability that a random walk on $T_n^{(d)}$ returns to its starting point after $2k$ steps, and likewise, $\tilde p_{T_n^{(d)}}(2k)$ the probability that a  non-backtracking random walk on $T_n^{(d)}$ returns to its starting point after $2k$ steps.  Then note that for $n > 2k$, it is clear that $p_{T_n^{(d)}}(2k) = p(2k)$ and that $\tilde p_{T_n^{(d)}}(2k) = \tilde p(2k)$. From here on the the proof, we will simply denote each of these by $p(2k)$ and $\tilde p(2k)$ respectively. More precisely, given a $k$, we will choose an $n_k > 2k$, and look at the walks on $T_{n_k}^{(d)}$.

Since $n_k> 2k$, then by symmetry, the probability of a random walk (or non-backtracking random walk) returning to its starting point after $2k$ steps is the same for any starting vertex in $T_n^{(d)}$.  Thus, the diagonal entries of $P^{2k}$ and $\tilde P^{(2k)}$ are all identical, and equal to $p(2k)$ and $\tilde p(2k)$ respectively.  Let $N_k$ denote the total number of vertices of $T_{n_k}^{(d)}$, and then we can express
\[
p(2k) = \frac{1}{N_k}\Tr(P^{2k}) = \frac{1}{N_k}\sum_{i=0}^{N_k} \lambda_i^{2k}
\]
where $\lambda_1,...,\lambda_{N_k}$ denote the eigenvalues of $P$, the transition probability matrix for $T_{n_k}^{(d)}$.  Since the simple random walk is transient, from Lemma \ref{lem:conv} we obtain
\begin{equation}\label{eq:eig_conv}
\sum_{k=0}^\infty\left(\frac{1}{N_k}\sum_{i=0}^{N_k}\lambda_i^{2k}\right) < \infty.
\end{equation}

Now, recall that the eigenvalues of $\tilde P^{(2k)}$ are $p_{2k}(\lambda_1),...,p_{2k}(\lambda_{N_k})$, where $p_{2k}(x)$ is the polynomial defined in the previous section.  In a similar manner to the above, we have
\[
\tilde p(2k) = \frac{1}{N_k}\Tr(\tilde P^{(2k)}) = \frac{1}{N_k}\sum_{i=0}^{N_k}p_{2k}(\lambda_i).
\]

In view of Lemma \ref{lem:eig_asympt}, we will consider eigenvalues above and below the threshold $2\sqrt{r-1}/r$.  Let us suppose that the eigenvalues are ordered so that $\lambda_0,...,\lambda_m$ have absolute value below $2\sqrt{r-1}/r$ and $\lambda_{m+1},...,\lambda_{N_k}$ those above $2\sqrt{r-1}/r$.  Then by Lemma \ref{lem:eig_asympt},recalling that our graph is $2d$-regular, we certainly have
\[
\tilde p(2k) = \frac{1}{N_k}\sum_{i=0}^m p_{2k}(\lambda_i) + \frac{1}{N_k}\sum_{i=m+1}^{N_k}p_{2k}(\lambda_i) \le C_{2d}k\left(\frac{1}{\sqrt{2d-1}}\right)^k + \frac{1}{N_k}\sum_{i=0}^{N_k}\lambda_i^{2k}.
\]
Therefore 
\[
\sum_{k=0}^\infty \tilde p(2k) \le C_{2d}\sum_{k=0}^\infty k\left(\frac{1}{\sqrt{2d-1}}\right)^k + \sum_{k=0}^\infty\left(\frac{1}{N_k}\sum_{i=0}^{N_k}\lambda_i^{2k}\right).
\]
The first term on the right is clearly a convergent series, and the second term is convergent by (\ref{eq:eig_conv}).  This gives the lemma.
\end{proof}

Applying Lemma \ref{lem:conv}, this proposition immediately gives us the following.
\begin{corollary}\label{cor:ge3}
A non-backtracking random walk on $\Z^d$ with $d\geq3$ is transient.
\end{corollary}

Thus, to finish the proof of Theorem \ref{thm:nbp}, we need only handle the case of $d=2$, which we will do in the next section.

We remark that this proof of Corollary \ref{cor:ge3} does not use the specific structure of the grid $\Z^d$.  Indeed, all we needed was Lemma \ref{lem:eig_asympt}, which applies to non-backtracking random walks on any regular graph, and the prior knowledge that a simple random walk on $\Z^d$ for $d\ge3$ is transient, and then the ability to examine the trace of a finite quotient of $\Z^d$, for which the probability would be the same for each vertex, allowing us to make the appropriate comparison.  Thus, we have actually proven the more general statement as follows.

\begin{proposition}
Let $G$ be any infinite vertex transitive graph with finite vertex transitive quotients of any given size.  Then if a simple random walk on $G$ is transient, then a non-backtracking random walk on $G$ is also transient.  
\end{proposition}

We remark further that it seems intuitively clear that 
\[
\tilde p(k) \leq p(k)
\]
for all $k$, in any infinite regular graph.  Numerical computation gives considerable evidence for this, and we strongly suspect that this is true.  However, the above is sufficient for our purposes.

\section{Recurrence on $\Z^2$}\label{sec:Z2}

We remark that our proof for transience of a non-backtracking random walk on $\Z^d$ when $d\ge3$ relied on the fact that the simple random walk is transient.  In the $d=2$ case, the simple random walk is recurrent, so that the series from Lemma \ref{lem:conv} diverges, and thus the comparison from Lemma \ref{lem:eig_asympt} is not informative.  Therefore we take a different approach that will more precisely nail down the asymptotic growth rate of $\tilde p(2k)$ in this case.

In section \ref{sub:polya}, in our sketch of the proof of P\'olya's Theorem, we enumerated the total number of closed simple walks on $\Z^2$ and found it to be
\[\binom{2n}{n}^2.\]
We therefore know that, if $A$ is the adjacency operator on $\Z^2$, then any diagonal entry of $A^{2n}$ is $\binom{2n}{n}^2$.  Thus if we wish to count the number of closed non-backtracking walks of length $2n$ on $\Z^2$ from a vertex to itself, then we can use Lemma \ref{lem:A^n}, setting $r=4$ since $\Z^2$ is $4$-regular, and obtain the diagonal entry of $\tilde A^{(2n)}$.

\begin{lemma}
The total number of closed non-backtracking walks of length $2n$ from a vertex to itself in $\Z^2$ is
\[
\sum_{i=0}^{n}(-3)^i\binom{2n-i}{i}\binom{2n-2i}{n-i}^2 - \sum_{i=0}^{n-1}(-3)^i\binom{2n-i-2}{i}\binom{2n-2i-2}{n-i-1}^2.
\]
Changing the indices, this can alternatively be expressed as
\[\sum_{k=0}^n (-3)^{n-k}\binom{n+k}{2k}\binom{2k}{k}^2 - \sum_{i=0}^{n-1}(-3)^{n-1-k}\binom{n+k-2}{2k-2}\binom{2(k-1)}{k-1}^2.
\]
\end{lemma}

It so happens that the expression $\displaystyle \sum_{k=0}^n (-3)^{n-k}\binom{n+k}{2k}\binom{2k}{k}^2$ shows up in the study of the \emph{central trinomial coefficients}, $T_n$, which are defined to be the largest coefficient in the expansion of $(1+x+x^2)^n$.  Formally, that is
\[
T_n = [x^n](1+x+x^2)^n.
\]
From the definition, one can derive the formula
\[
T_n = \sum_{k=0}^{\lfloor n/2 \rfloor}\binom{n}{2k}\binom{2k}{k}.
\]

In a paper of Zhi-Wei Sun (\cite{sun}), it is proven that $T_n$ satisfy the following relationship with the above sum.
\begin{lemma}[Lemma 4.1 of \cite{sun}]
For any $n\in \N$ we have
\[
T_n^2 = \sum_{k=0}^n\binom{n+k}{2k}\binom{2k}{k}^2(-3)^{n-k}.
\]
\end{lemma}
The proof in \cite{sun} involves the manipulation of identities from the theory of hypergeometric series.

From this we obtain an expression for the number of closed walks from a vertex to itself on $\Z^2$ in terms of the squares of the central trinomial coefficients.

\begin{corollary}\label{cor:diffsquares}
For any $n\in\N$ and any vertex $v\in\Z^2$, we have
\[
\tilde A^{(2n)}(v,v) = T_n^2 - T_{n-1}^2.
\]
\end{corollary}

The asymptotics of the numbers $T_n$ are investigated in \cite{wagner} using singularity analysis of the generating function for a generalization of the numbers $T_n$.  A special case of their main result gives the following.

\begin{lemma}[\cite{wagner}]\label{lem:trinasymp}
The asymptotics for the numbers $T_n$ are given by
\[
T_n = \frac{\sqrt3}{2\sqrt{n\pi}}3^n\left(1 - \frac{3}{16n} + O\left(\frac{1}{n^2}\right)\right).
\]
\end{lemma}

\begin{corollary}\label{cor:asymp}
Asymptotically, the number of closed non-backtracking walks from a vertex to itself on the grid $\Z^2$ is given by
\[
\tilde A^{(2n)}(v,v) \sim \frac{2}{\pi n}3^{2n-1}.
\]
\end{corollary}

\begin{proof}
Using Corollary \ref{cor:diffsquares} and Lemma \ref{lem:trinasymp}, we have
\[\begin{split}
\tilde A^{(2n)}(v,v) &= T_n^2 - T_{n-1}^2 \\&= \frac{3}{4\pi n}3^{2n}\left(1 + O\left(\frac1n\right)\right) - \frac{3}{4\pi(n-1)}3^{2n-2}\left(1 + O\left(\frac1n\right)\right) \\&= \frac{24n-27}{4\pi n(n-1)}3^{2n-2}\left(1 + O\left(\frac1n\right)\right) \\&= \frac{2}{\pi n}3^{2n-1}\left(1 + O\left(\frac1n\right)\right)
\end{split}\]
and the result follows.
\end{proof}

We are now ready to show recurrence for a non-backtracking random walk on $\Z^2$.

\begin{corollary}
A non-backtracking random walk on the infinite grid $\Z^2$ is recurrent.
\end{corollary}
\begin{proof}
Let $\tilde p(k)$ denote the probability that a non-backtracking random walk on $\Z^2$ returns to its starting point after $k$ steps.  Note that the total number of non-backtracking random walks of length $k$ is
\[
4\cdot 3^{k-1}
\]
since there are 4 choices for the first step, and then 3 choices for each subsequent step since we must exclude the edge that would backtrack.  Note also that $\tilde p(k) = 0 $ for $k$ odd since a walk on $\Z^2$ returning to its starting point must contain an equal number of steps up as down, and an equal number of steps to the left as to the right.  So we need only consider $\tilde p(2k)$. The total number of non-backtracking walks of length $2k$ returning to their starting vertex $v$ is given by $\tilde A^{(2k)}(v,v)$, so by Corollary \ref{cor:asymp}, we obtain $\tilde p(2k)$ asymptotically is given by
\[
\tilde p(2k) \sim \frac{\frac{2}{\pi k}3^{2k-1}}{4\cdot 3^{2k-1}} = \frac{1}{2\pi k}.
\]
Therefore \[\sum_{k=0}^\infty \tilde p(2k) = \sum_{k=0}^\infty \frac{1}{2\pi k}\] which is divergent.  Therefore, by Lemma \ref{lem:conv}, the random walk is recurrent.
\end{proof}

With this, we have finished the proof of Theorem \ref{thm:nbp}.\\

As a final remark, we observe that Corollary \ref{cor:diffsquares} gives an exact enumeration of closed non-backtracking walks on $\Z^2$ in terms of a well-known combinatorial object, the central trinomial coefficients. There are numerous sets enumerated by central trinomial coefficients.  For example, $T_n$ counts the number of length $n$ paths from the point $(0,0)$ to the point $(n,0)$ that take steps to the right, diagonally up and to the right, or diagonally down and to the right.  Thus, this suggests some kind of combinatorial connection between such paths and closed non-backtracking walks on the $\Z^2$.  It would be of interest to come up with a bijective proof of this fact, or to otherwise illuminate this combinatorial connection. \\

{\bf Acknowledgment.} The author would like to thank Fan Chung for the initial suggestion to think about P\'olya's Theorem in the context of non-backtracking random walks, and for several helpful discussions throughout the process of writing this paper.



\begin{thebibliography}{}

\bibitem{alon}
N. Alon, I. Benjamini, E. Lubetzky, and S. Sodin, Non-backtracking random walks mix faster, {\it Communications in Contemporary Mathematics}, 9 (2007) 585-603.

\bibitem{alon2}
N. Alon and E. Lubetzky. Poisson approximation for non-backtracking random walk, {\it Israel J. Math.}, 174 (2009), 227-252.

\bibitem{angel}
O. Angel, J. Friedman, and S. Hoory, The non-backtracking spectrum of the universal cover of a graph, {\it Transactions of the American Mathematical Society},  32 (2015), no. 6, 4287-4318.




\bibitem{cioaba}
S. Cioaba and P. Xu, Mixing rates of random walks with little backtracking, {\it SCHOLAR--a scientific celebration highlighting open lines of arithmetic research}, 27-58, Contemp. Math., 655, {\it Amer. Math. Soc., Providence, RI}, 2015.

\bibitem{doyle}
P.G. Doyle, {\it Applications of Rayleigh's short-cut method to P\'olya's recurrence problem}, PhD thesis, Dartmouth College, 1982. Online at http://www.math.dartmouth.edu/~doyle/docs/thesis/thesis.pdf.

\bibitem{doylesnell}
P.G. Doyle and J.L. Snell, {\it Random Walks and Electric Networks}, The Mathematical Association of America, 1984.

\bibitem{prob}
R. Durrett, {\it Probability: Theory and Examples}, Cambridge University Press, 2013.

\bibitem{Fitzner}
R. Fitzner and R. van der Hofstad, Non-backtracking random walk, {\it J. Stat. Phys.}, 150 (2013), no. 2, 264-284.

\bibitem{ihara}
M. Kempton, Non-backtracking random walks on graphs and a weighted Ihara's theorem, {\it Open Journal of Discrete Mathematics}, 6 (2016), no. 4, 207-226.

\bibitem{redemption}
F. Krzakala, C. Moore, E. Mossel, J. Neeman, A. Sly, L. Zdeborova, and P. Zhang, Spectral redemption in clustering sparse networks, {\it Proceedings of the National Academy of Sciences},  110 (2013), no. 52, 20935-20940.

\bibitem{lovasz}
L. Lov\'asz, Random walks on graphs: a survey, {\it Combinatorics, Paul Erd\H os is Eighty (Volume 2),} Keszthely (Hungary) (1993) 1-46.

\bibitem{moore}
C. Moore, Random walks, {\it Ramanujan Mathematical Society Mathematics Newsletter}, 17 (2007), no. 3, 78-84.

\bibitem{novak}
J. Novak, P\'olya's random walk theorem, {\it Amer. Math. Monthly}, 121 (2014), no. 8, 711-716.

\bibitem{polya}
G. P\'olya, \"Uber eine aufgabe betreffend die irrfahrt im strassennetz, {\it Math. Ann.}, 84 (1921), 149-160.

\bibitem{sabot}
C. Sabot and P. Tarres, Edge-reinforced random walk, vertex-reinforced jump process and the supersymmetric hyperbolic sigma model, {\it J. Eur. Math. Soc.}, 17 (2015), no. 9, 2353-2378.

\bibitem{asymptopia}
J. Spencer and L. Florescu, {\it Asymptopia}, Student Mathematical Library, Volume 71, AMS Publications, 2014.

\bibitem{sun}
Z.-W. Sun, Congruences involving generalized central trinomial coefficients, {\it Sci. China Math.}, 57 (2014), no. 7, 1375-1400.

\bibitem{tetali}
P. Tetali, Random walks and effective resistance of networks, {\it Journal of Theoretical Probability}, 4 (1991), 101-109.

\bibitem{wagner}
S. Wagner, Asymptotics of generalized trinomial coefficients, 2012, Available online at http://arxiv.org/abs/1205.5402v3.

\end{thebibliography}
\end{document}